\documentclass[reqno]{amsart}
\usepackage{amssymb, amsmath, amsthm, amscd, mathrsfs}
\usepackage{verbatim}
\usepackage{a4wide}
\usepackage{bbm}
\usepackage{enumerate}
\usepackage{color}
\allowdisplaybreaks
\usepackage{ulem}
\usepackage[colorlinks]{hyperref}
\usepackage[colorinlistoftodos]{todonotes}
\usepackage{verbatim}
\usepackage{hyperref}
\theoremstyle{plain}
\newtheorem{theorem}{Theorem}[section]
\theoremstyle{remark}

\theoremstyle{plain}

\newtheorem{proposition}[theorem]{Proposition}

\numberwithin{equation}{section}
\def\N{{\mathbb N}}

\def\R{{\mathbb R}}

\def\C{{\mathbb C}}

\newcommand{\ds}{\displaystyle}
\newcommand\numberthis{\addtocounter{equation}{1}\tag{\theequation}}

\begin{document}
\author{E. M. Ait Ben Hassi}
\address{E. M. Ait Ben hassi, Cadi Ayyad University, Faculty of Sciences Semlalia, 2390, Marrakesh, Morocco}
\email{m.benhassi@uca.ma}	

\author{M.  Fadili}
\address{M. Fadili, Cadi Ayyad University, Faculty of Sciences Semlalia, 2390, Marrakesh, Morocco}
\email{fadilimed@live.fr}

\author{L.  Maniar}
\address{L. maniar, Cadi Ayyad University, Faculty of Sciences Semlalia, 2390, Marrakesh, Morocco}
\email{maniar@uca.ma}

\title[]{ Controllability of a system of degenerate parabolic equations with non-diagonalizable 
	diffusion matrix}



\begin{abstract}
	In this paper we  study the null controllability  of some non diagonalizable degenerate parabolic systems of PDEs, we assume that the diffusion, coupling and controls matrices are constant and we characterize the null controllability by an algebraic condition so called   \textit{Kalman's rank} condition. 
\end{abstract}

\maketitle

	
\section{Introduction and Main result }

In this paper we focus on the controllability properties of some non-diagonalizable parabolic degenerate systems.
\begin{equation}\label{syst24}
\left\lbrace \begin{array}{lll}
\partial_t Y=(\mathbf{D}\mathcal{M}   +A) Y + B v \mathbbm{1}_{\omega}& in & Q, \\
\mathbf{C} Y=0   & on & \Sigma, \\
Y(0,x)=Y_0(x)\,\,  & in & (0,1),
\end{array}
\right.
\end{equation}
where $Q:=(0,T)\times(0,1)$, $\Sigma~:~=~(0,T)~\times~\{0,1\}$, for $T>0$ and $\omega\subset (0,1)$ a is a (small) nonempty open control region,  $ \mathbbm{1}_{\omega}$ denotes the characteristic function of $\omega$. The diffusion matrix $\mathbf{D}$ is a non-diagonalizable $n\times n$ matrix that satisfies the following assumptions :
\begin{itemize}
	\item there exists $\alpha_0>0$ such that \begin{equation}\label{cond1}
	  \mathbf{D}\xi. \xi\geqslant \alpha_0|\xi|^2\quad \forall \xi \in \R^n
	\end{equation}
	\item there exists a non-singular matrix \begin{equation}\label{h1} P\in \mathcal{L}(\C^n)\text{  such that } \mathbf{D}=PJP^{-1}\end{equation} for some $J\in \mathcal{L}(\C^n)$ of the form
	\begin{equation*}\label{h2} J=\mathrm{ diag} (J_1,\cdots,J_p),\end{equation*}
	where the $J_i$ are the Jordan blocks associated to the eigenvalues  $d_i$ of $\mathbf{D}$.
	\begin{equation}\label{jordan:block}
	J_i= \begin{bmatrix} 
	d_i & 1 & & \\
	& \ddots & \ddots & \\
	& & \ddots& 1\\
	& & & d_i 
	\end{bmatrix}
	\end{equation}
	with $\mathrm{R}e\, d_i>0$.\\
\end{itemize}

 The coupling matrix $A$ is a $n\times n$ constant matrix and the control matrix $B$ is a $n\times m$ constant matrix. The operator $\mathcal{M}$ is defined by  $  \mathcal{M} y= {\left(  a y_x \right)}_x$ for $y\in D(\mathcal{M} )\subset L^2(0,1)$.  For  $Y=(y_1,\cdots,y_n)^*$,  $\mathcal{M}Y$ denotes $(\mathcal{M}y_1,\cdots,\mathcal{M}y_n)^*$. 
The function  $a$ is a diffusion coefficient which degenerates at $0$ (i.e., $a(0)=0$)  and which can be either  weak  degenerate (WD), i.e.,
\begin{equation}
\text{(WD)}\,\,	
\begin{cases}
(i)\,\,a \in \mathcal{C}([0,1])\cup\mathcal{C}^1((0,1]),\, a>0 \text{ in }(0,1],\,a(0)=0,\\
(ii)\,\,\exists K\in[0,1)\text{ such that }xa'(x)\leqslant Ka(x), \;\,\forall x\in[0,1],
\end{cases}
\end{equation}
or  strong degenerate (SD), i.e.,
\begin{equation}
\text{(SD)}\,\,
\begin{cases}
(i)\,\,a \in \mathcal{C}^1([0,1]),\, a>0 \text{ in }(0,1],\,a(0)=0,\\
(ii)\,\,\exists K\in[1,2)\text{ such that }xa'(x)\leqslant Ka(x)\,\forall x\in[0,1], \\
(iii)\begin{cases}
\ds\exists\theta\in(1,K] x\mapsto\frac{a(x)}{x^{\theta}}\text{ is nondecreasing near }0, \text{ if }K>1,\\
\ds\exists\theta\in(0,1) x\mapsto\frac{a(x)}{x^{\theta}}\text{ is nondecreasing near }0, \text{ if }K=1.
\end{cases}
\end{cases}
\end{equation}
The boundary  condition  $\mathbf{C}Y=0$ is either  $Y(0)=Y(1)=0$ in the weak degenerate  case $(WD)$ or
$Y(1)=(aY_x)(0)=0$ in the strongly degenerate case $(SD)$.

It will be said that \eqref{syst24} {\it  is null-controllable } at time $T$ if, for any $Y_0 \in L^2((0,1))^n$,  there exists $v \in L^2(\omega\times(0,T))^m$ such that the associated solution satisfies
\begin{equation*}
y(T,x) = 0 \text{ in } (0,1).	 
\end{equation*}

The system \eqref{syst24}  is said to be {\it approximately controllable } at time $T$ if, for any $Y_0, Y_1 \in L^2((0,1))^n$ and $\varepsilon>0$, there exists $v \in L^2(\omega\times(0,T))^m$  such that the solution of \eqref{syst24}  satisfies
\begin{equation*}
\|y(T,x)-Y_1\|_{L^2(0,1)^n} \leqslant \varepsilon. 	 
\end{equation*}

Controlling coupling systems of partial differential equations attracted growing interest during the last decade, the main question is whether it is possible to control such systems with fewer controls (i.e the number of controls is less than the number of equations). For finite dimensional linear systems, the controllability can be characterized by algebraic rank
condition on the matrices generating the dynamics and taking account of the control action.  The
theory has been adapted and extended to more general systems including infinite dimensional
systems. At our knowledge, in the nondegenerate case,  M. Gonzalez-Burgos, L. de Teresa  \cite{GonTer}  provided a null controllability
result for a cascade parabolic system by one control force under a condition on the sub-diagonal of the coupling matrix.   F. Ammar-Khodja et al. \cite{A2,A3}  obtained several results characterizing the null controllability  of fully coupled systems with $m$-control forces by a generalized Kalman rank condition. In \cite{FerGonTer}, the authors gave controllability results for a system in the case where the diffusion matrix is non diagonalizable.  

For degenerate systems,  the case  of two coupled equations ($n=2$),  cascade  systems  are considered  in \cite{CaMaVa2,de-Ca} and  in \cite{bahm, hjjaj}  the authors have studied the null controllability of degenerate noncascade  parabolic systems. 

In  \cite{FadiliManiar}, we have extended the null controllability results obtained by Ammar-Khodja et al. \cite{A3} to a class of parabolic degenerate systems of PDEs in the two following cases 
\begin{enumerate}
	\item the coupling matrix $A$ is a cascade one and the diffusion matrix \\ $\mathbf{D}=diag(d_1,\cdots,d_n)$ where $d_i>0$, $i=1,\cdots,n$,
	\item the coupling matrix $A$ is a full matrix (noncascade) and the diffusion matrix $\mathbf{D}=d I_n, \; d>0$.
\end{enumerate}
On the other hand, in \cite{benfama} we study the null controllability of \eqref{syst24} in the case where de diffusion matrix $\mathbf{D}$ is diagonalizable $n\times n$ matrix with positive real eigenvalues, i.e.,
\begin{equation}\label{D:diagonalizable}
\mathbf{D}=P^{-1} \mathbf{J} P,\,\, P\in \mathcal{L}(\R^n),\,\, det(P)\neq 0,
\end{equation}
where  $\mathbf{J}=diag(d_1,\cdots,d_n)$, $d_i>0, 1\leqslant i\leqslant n$

In the current  paper, we assume that diffusion matrix $\mathbf{D}$ is non-diagonalizable. We use the same approach as \cite{FerGonTer} without imposing that Jordan's block sizes are bounded by  4. Thus our proof is also an improvement of the one given in \cite{FerGonTer} for the nondegenerate case.

Notice that, if \eqref{cond1} is satisfied, for every $v\in L^2(\omega\times(0,T); \R^m)$ and every $y_0 \in L^2((0,1);\R^n)$, \eqref{syst24} possesses a unique weak solution $y$, with
$y\in C\left([0,T],\left(L^2(0,1)\right)^n\right)\cap L^2\left(0,T;{\left( H_{a}^1\right)}^n \right)$
see Section 2.

In order to study the null controllability of system \eqref{syst24},  we will
consider the following corresponding adjoint problem
\begin{equation}\label{adjsyst24}
\left\lbrace \begin{array}{lll}
-z_t-\mathbf{D}^*\mathcal{M}z =  A^* z & in & Q, \\
\mathbf{C} z=0   & on & (0,T), \\
z(T,x)=z_T(x)\,\,  & in & (0,1).
\end{array}
\right.
\end{equation}

Since  the null controllability of system \eqref{syst24} is equivalent to the existence of a positive constant $C$ such that, for every $z_0\in L^2(0,1)^n$, the solution $z\in\mathcal{C}^{0}([0,T];L^2(0,1)^n)$  to the adjoint system \eqref{adjsyst24} satisfies the observability inequality :
\begin{equation}\label{observ:inequality}
	\|z(\cdot,0) \|^2_{L^2(0,1)^n}\leqslant C \int\!\!\!\!\!\int_{\omega\times(0,T)} |B^*z(x,t)|^2,
\end{equation}

The strategy used in this case is slightly different from the one used in \cite{benfama}, although in both cases it is necessary to show Carleman estimates for a scalar PDE of order 2n in space.

Note that this technique  failed in the case where $\mathcal{M}$ is a compact operator, since its spectrum admits zero as a point of accumulation, so the Kalman condition is no longer verified, which does not ensure a perfect coupling of the equations.

All along the article, we use generic constants for the estimates, whose values may change from
	line to line.

Let us remark that when $A\in\mathcal{L}(\R^n)$ and $B\in\mathcal{L}(\R^m,\R^n)$ are constant matrices, $[A|B] \in\mathcal{L}(\R^{nm},\R^n)$ is the matrix given by
\begin{equation*}
[A|B] = \big(A^{n-1}B|A^{n-2}B|\cdots|AB|B\big) 
\end{equation*}
With this notation, we have the following main result.
 
	\begin{theorem}\label{main result}
		Let assume that $\mathbf{D},\, A \in \mathcal{L}(\R^n)$, $B\in\mathcal{L}(\R^m;\R^n)$  such that $\mathbf{D}$ satisfies \eqref{cond1}-\eqref{jordan:block}. Then \eqref{syst24} is null controllable at time $T$ if and only if 
		\begin{equation}\label{condiequivalente}
		\mathrm{rank} [\lambda_i\mathbf{D}-A|B]=n\:\:\forall i\geqslant 1.	
		\end{equation}
	\end{theorem}  

The rest of this paper is organized as follows. In Section 2  prove the wellposedness of the problem \eqref{syst24}.  Section 3 is devoted to some controllability results for one parabolic equation. In Section 4,  we prove some useful estimates  under the assumption  \eqref{condiequivalente} and we establish  Carleman estimates for a scalar PDE of order $2n$ in space. In Section 5, we give the proof of the main result. And finally, in Section 6 we study the  null controllability for semilinear systems.
\section{ Wellposedness of the problem}
The semigroup generated by the operator $(\mathcal{M},D(\mathcal{M}))$ is analytic with
angle $\frac{\pi}{2}$  ( see \cite[Theorem 2.8]{cmp}).
In order to prove the wellposedness of the problem \eqref{syst24}, it suffices to show that the operator $\mathbf{D}\mathcal{M}$ generates a $c_0$-semigroup. in fact, similarly like in \cite{Oliveira} we prove, under the assumption that all eigenvalues of the diffusion matrix $\mathbf{D}$ have positive real part the operator $\mathbf{D}\mathcal{M}$ is the generator of an analytic semigroup.

\begin{equation*}
{\small  H_{a}^1=\{ u \in L^2(0,1)/  u \text{ absolutely continuous in} [0,1], \sqrt{a}u_x \in L^2(0,1) \text{ and } u(1)=u(0)=0 \}	}
\end{equation*}
and 
$$\displaystyle H_{a}^2=\left\lbrace u \in H_{a}^1(0,1)/ au_x \in H^1(0,1)  \right\rbrace.
$$
in the (WD) case and 
\begin{equation*}
 H_{a}^1=\{ u \in L^2(0,1)/  u \text{   absolutely continuous in} (0,1], \sqrt{a}u_x \in L^2(0,1) \text{ and } u(1)=0 \}
\end{equation*}
 and 
\begin{align*}
\displaystyle H_{a}^2 &=\{ u \in H_{a}^1(0,1)/ au_x \in H^1(0,1)  \}\\
&=\{ u \in L^2(0,1)/  u \text{   absolutely continuous in} (0,1], au \in H_0^1(0,1),au_x \in H^1(0,1)\text{ and } (au_x)(0)=0  \},
\end{align*}
in the (SD) case with the norms
$${\parallel u\parallel}_{H_{a}^1}^2 ={\parallel u\parallel}_{L^2(0,1)}^2+{\parallel \sqrt{a} u_x\parallel}_{L^2(0,1)}^2, \quad{\parallel u\parallel}_{H_{a}^2}^2 ={\parallel u\parallel}_{H_{a}^1}^2+{\parallel {(a u_x)}_x\parallel}_{L^2(0,1)}^2.$$

We recall also some results about the spectrum of the operator $-\mathcal{M}$ already used in  \cite{benfama}, indeed,  the operator $\mathcal{-M}$ is a definite positive operator. Knowing that $D(\mathcal{M})=H^2_a(0,1)$ is compactly embedded in $L^2(0,1)$ see \cite{CanFrag}. Thus, the spectrum of $-\mathcal{M}$ consists of eigenvalues 
\begin{equation}
	0<\lambda_1<\lambda_2<\dots<\lambda_j<\cdots \qquad\text{ with } \qquad \lambda_j\longrightarrow +\infty
\end{equation}

Therefore : There exists a complete orthonormal set $\{\mathbf{w}_{j}\}$ of eigenvectors of $-\mathcal{M}$.\\
For all $z\in D(-\mathcal{M})$ we have 
\begin{equation}
-\mathcal{M}  z=\sum_{j=1}^{+\infty} \lambda_j  \langle z ,\mathbf{w}_{j} \rangle \mathbf{w}_{j} = \sum_{j=1}^{+\infty} \lambda_j E_j z,
\end{equation}
where $\langle \cdot,\cdot\rangle $ is the inner product in $L^2(0,1)$ and 
\begin{equation*}
	 E_j z= \langle z ,\mathbf{w}_{j} \rangle \mathbf{w}_{j}.
\end{equation*}  
  So, $\{E_j\}$ is a family of complete orthogonal projections in $L^2(0,1)$ and
  $$z= \sum_{j=1}^{+\infty} E_j z,\quad z\in L^2(0,1) .$$
  
  $\mathcal{-M}$  generates a strongly continuous semigroup $\{e^{\mathcal{-M}t }\}$ given by
  $$ e^{\mathcal{-M}t }z=\sum_{j=1}^{+\infty}e^{\lambda_j t } E_j z.$$

 In the Hilbert space  $\mathbb{H}:= (L^2(0,1))^n$we define the following linear operator:
 $\ds 
 \mathcal{A}: D(\mathcal{A})\subset \mathbb{H}\rightarrow \mathbb{H}
 $ given by  $D(\mathcal{A})=(H_{a}^2)^n$ and $\mathcal{A} u=-\mathbf{D}\mathcal{M} u$.
 
 We have the following result. 
 \begin{theorem}
 	  Assume that all eigenvalues of $\mathbf{D}$ have positive real part. Then $\mathcal{A}$
 	 is sectorial and therefore, $-\mathcal{A}$ is the generator of an analytic semigroup $\{e^{-\mathcal{A}t}\,:\,t\geq 0\}$ in $\mathbb{H}$.
 \end{theorem}
 \begin{proof}
 	  Let $\theta\in (0,\frac{\pi}{2})$  such that $|\mathrm{arg}\lambda_{\mathbf{D}}|<\theta$  for any eigenvalue $\lambda_{\mathbf{D}}$ of $\mathbf{D}$. We prove that the sector
\begin{equation*}
	 S=\{z\in\C\,:\:\theta\leqslant |\mathrm{arg}z|\leqslant \pi,\,z\neq 0  \}
\end{equation*} 	 
is in the resolvent set of $\mathcal{A}$  and there exists a constant $C$ such that for any $z\in S$ 	 
 	 $$ \|(z-\mathcal{A})^{-1} \|\leqslant \frac{C}{|z|}$$
For $z\in S$ and  $f\in \mathbb{H}$, let u be given by 	
$$u=\sum_{j=1}^{\infty} (z-\lambda_j\mathbf{D})^{-1}f_{j}\mathbf{w}_{j} $$ 
 where $f_{j}=\int_{(0,1)}f\mathbf{w}_{j}dx$. Since $z\in S$  implies $\ds\frac{z}{\lambda_j}$  is not an eigenvalue of $\mathbf{D}$, the matrix $z-\lambda_j\mathbf{D}$ is invertible and there exists a constant $C>0$ such that 
 $\ds \|(z-\lambda_j\mathbf{D})^{-1} \|\leqslant \frac{C}{|z|}$ for all $ j\geqslant 1.$
  It follows that the above series is convergent in $\mathbb{H}$, so $u$ is well defined and  
  $\|u\|\leqslant \frac{C}{|z|}\|f\|$.  Also,
  \begin{align*}
  	zu+\mathbf{D}\mathcal{M}u&=\sum_{j=1}^{\infty}\big[ z(z-\lambda_j\mathbf{D})^{-1}f_{j}-\lambda_j\mathbf{D}f_{j}\big]\mathbf{w}_{j}\\
  	&=\sum_{j=1}^{\infty}(z-\lambda_j\mathbf{D})^{-1}\big[ z-\lambda_j\mathbf{D}\big]f_{j}\mathbf{w}_{j}\\
 	&=\sum_{j=1}^{\infty} f_{j}\mathbf{w}_{j}= f,
   \end{align*}
  so, $u= (z-\mathcal{A})^{-1}f$. Therefore, $z$ is in the resolvent set of $\mathcal{A}$,
   $$ \|(z-\mathcal{A})^{-1} \|\leqslant \frac{C}{|z|}$$
  and the proof is complete.
 \end{proof}

Thus, \eqref{syst24} is well posed in the sens of semigroup theory and the following global existence result holds.
\begin{theorem}\label{theo2.2}
	 Under the Hypothesis \eqref{h1}-\eqref{jordan:block}, for all $(y_1^0,\cdots,,y_n^0)\in (L^2(0,1))^n$ and  $v\in (L^2(Q))^m$  
	 there exists a unique  mild solution $(y_1(t),\cdots,y_n(t))$ of \eqref{syst24} which belongs to $$X_T:=C\left([0,T],\left(L^2(0,1)\right)^n\right)\cap L^2\left(0,T;{\left( H_{a}^1\right)}^n \right)$$ and satisfies
	 \begin{multline*}
	 \sup_{t\in[0,T]}\|(y_1,\cdots,y_n)(t)\|^2_{\left(L^2(0,1)\right)^n }+\int_0^T \|
	 (\sqrt{a} y_{1 x},\cdots,\sqrt{a}y_{n x}) \|^2_{L^2}dt \\
	  \leq C_T\left(\| (y_1^0,\cdots,,y_n^0) \|^2_{(L^2(0,1))^n}+ \|v\|^2_{(L^2(Q))^m} \right),
	 \end{multline*}
	  	for a constant $C_T>0$.\\
	 Moreover, if $(y_1^0,\cdots,,y_n^0)\in \left( H^1_{a}\right)^n$,  then  $$(y_1,\cdots,y_n)\in C\left([0,T],\left( H^1_{a}\right)^n\right)
	 \cap\, H^1\left(0,T;\left(L^2(0,1)\right)^n\right)\cap L^2\left(0,T;\left( H^2_{a}\right)^n \right )$$
	 and
	 \begin{multline*}
	 \sup_{t\in[0,T]}\|(y_1,\cdots,y_n)(t)\|^2_{\left(H^1_{a}\right)^n} + \int_0^T\left( \| (y_{1 t},\cdots,y_{n t})\|^2_{L^2}+\|((a y_{1 x})_x,\cdots,(a y_{n x})_x)\|^2_{L^2} \right)dt\\
	 \leq C_T\left(\|(y_1^0,\cdots,y_n^0) \|^2_{\left( H^1_{a}\right)^n}+ \|v\|^2_{(L^2(Q))^m}
	 \right)\numberthis\label{therem22equ:2}
	 \end{multline*}
	 for a constant $C_T>0$.
\end{theorem}

\section{ Carleman estimates for one equation}

In order to establish a Carleman estimate for the adjoint system \eqref{adjsyst24}, we are led  to see Carleman estimates already established in the case of one single parabolic degenerate equation of order 2 in space \cite{benfama,FadiliManiar}.
\begin{equation}\label{problem}
\begin{cases}
u_t-(a(x)u_x)_x+ bu=f,\quad (t,x)\in Q, \\
\mathbf{C} u =0 \\
u(0,x)=u_0(x),   \quad x\in (0,1).
\end{cases}
\end{equation}
For this purpose, let us consider the following time and space weight functions
\begin{equation}\label{fonctionspoids}
\begin{cases}
\ds \theta(t)=\frac{1}{t^4(T-t)^4}\,\,, \psi(x)=\lambda\left(\int_0^x\frac{y}{a(y)}dy-c\right)\,\,\text{ and }\, \varphi(t,x)=\theta(t)\psi(x),\\
\Phi(t,x)=\theta(t)\Psi(x)\,\,  \,\text{ and  }\,\, \Psi(x)=e^{\rho\sigma(x)} -e^{2\rho{\parallel\sigma\parallel}_{\infty}},
\end{cases}
\end{equation}
where   $\sigma$  a $\mathcal{C}^2([0,1])$ function such that $\sigma(x)>0$ in $(0,1)$, $\sigma(0)=\sigma(1)=0$ and $\sigma_x(x)\neq 0$ in $[0,1]\setminus \omega_0$, $\omega_0$ is an open subset of $\omega$, and the parameters $c$, $\rho$ and $\lambda$ are  chosen as in \cite{FadiliManiar} such that

\begin{equation}\label{poids11} 
\ds c>4^n c_0 \,\,,\quad \,\, \rho>\frac{\ln\left(\frac{4^n(c-c_0)}{c-4^n c_0}\right)}{{\parallel\sigma\parallel}_{\infty}},
\end{equation}
\begin{equation}\label{poids12}
\frac{e^{2\rho{\parallel\sigma\parallel}_{\infty}}}{c-c_0}<\lambda<\frac{4^n}{(4^n-1)c}{\left( e^{2\rho{\parallel\sigma\parallel}_{\infty}}-e^{\rho{\parallel\sigma\parallel}_{\infty}}\right)}.
\end{equation} 
where $\ds c_0=\int_0^1\frac{x}{a(x)}dx$.

Let $\omega'$ a subset of $\omega$ and set $\omega'':=(x_1^{''},x_2^{''}) \subset\subset\omega'$ and $\xi\in \mathcal{C}^\infty([0,1])$ such that $0\leq \xi(x)\leq 1$
for $x\in(0,1)$, $\xi(x)=1$ for $x\in (0,x_1^{''})$ and $\xi(x)=0$
for $x\in (x_2^{''}, 1)$.\\
The two following results have been proved in \cite{benfama}. 
\begin{proposition}\label{propo3.6}
	Let $T>0$ and  $\tau\in \mathbb{R}$. Then there exists two positive constants $C$ and $s_0$ such that, for all $\ds u_0\in L^2(0,1)$, the solution $u$ of equation \eqref{problem} satisfies
	\small{\begin{multline}\label{eqproposition23}
		\int\!\!\!\!\!\int_{Q} \left(s^{\tau-1}\theta^{\tau-1}\xi^2 u_t^2+s^{\tau-1}\theta^{\tau-1}\xi^2 (\mathcal{M}u)^2 +s^{1+\tau}\theta^{1+\tau} a \xi^2 u_x^2+s^{3+\tau} {\theta}^{3+\tau}\frac{x^2}{a}\xi^2  u^2\right)e^{2s\varphi(t,x)}dxdt \\
		\leqslant C\left(\int\!\!\!\!\!\int_{Q} \xi^2 s^\tau\theta^{\tau}f^2(t,x)e^{2s\varphi(t,x)}dx dt+\int\!\!\!\!\!\int_{Q_{\omega'}} s^{2+\tau}\theta^{2+\tau} u^2 e^{2s\varphi(t,x)}dxdt  \right)
		\end{multline}}
	for all $s\geqslant s_0$.
\end{proposition}

Proposition \ref{propo3.6} gave a Carleman estimate in $(0, x_1^{'} )$. the following Proposition is a non degenerate Carleman estimate to the equation \eqref{problem} on the interval $(x_1^{'},1)$.
\begin{proposition}\label{proposition26}
	Let $T>0$ and $\tau\in\R$. Then, there exist two
	positive constants $C$ and $s_0$ such that for every $u_0\in L^2(0,1)$, the solution $u$ of equation  \eqref{problem} satisfies
	\begin{multline}\label{eqproposition26}
	\int\!\!\!\!\!\int_{Q} \Big(s^{\tau-1}\theta^{\tau-1}\varsigma^2 u_t^2+s^{\tau-1}\theta^{\tau-1}\varsigma^2 (\mathcal{M}u)^2+s^{1+\tau}\theta^{1+\tau} a\varsigma^2 u_x^2+s^{3+\tau} {\theta}^{3+\tau}\frac{x^2}{a}\varsigma^2 u^2 \Big)e^{2s\Phi}dx dt \\
	\leq C\Big(\int\!\!\!\!\!\int_{Q}  \varsigma^2 s^\tau \theta^{\tau}f^2(t,x)e^{2s\Phi}dx dt+\int\!\!\!\!\!\int_{Q_{\omega'}} s^{3+\tau}\theta^{3+\tau} u^2 e^{2s\Phi}  dx dt\Big)
	\end{multline}
	for all $s\geqslant s_0$, 	with $\varsigma=1-\xi$.
\end{proposition}

\section{Some useful results}

Since the number of control forces is less than the number of equations, we need to highlight the equation coupling tools. In deed the equations are coupled by  means this algebraic condition
 $\ds  \mathrm{rank} [\lambda_i\mathbf{D}-A|B]=n\:\:\forall i\geqslant 1$. Let us introduce the following operators $\ds \mathcal{K} :D(\mathcal{K})\subset L^2([0,1];\R^{nm})\mapsto L^2([0,1];\R^{n})$ and $\ds \mathcal{K}_{*} :D(\mathcal{K}_{*})\subset L^2([0,1];\R^{n})\mapsto L^2([0,1];\R^{nm})$ with 
$$D(\mathcal{K})=\{v\in L^2([0,1];\R^{nm}) : [-\mathbf{D}\mathcal{M}-A|B]v\in L^2([0,1];\R^{n}) \} $$
$$D(\mathcal{K}_{*})=\{\varphi\in L^2([0,1];\R^{n}) : [-\mathbf{D}\mathcal{M}-A|B]^{*}\varphi\in L^2([0,1];\R^{nm}) \} $$
and 
\begin{equation}\label{operteurdekalman}
	 \mathcal{K}v:=[-\mathbf{D}\mathcal{M}-A|B]v,\:\: \mathcal{K}_{*}\varphi:= [-\mathbf{D}\mathcal{M}-A|B]^{*}\varphi 
\end{equation}
$\ds \mathcal{K}$ and $\ds \mathcal{K}_{*}$ are densely defined unbounded operators.  we have the following estimate

\begin{proposition}\label{prpo2.2}
	 Let us assume that $\ds  \mathrm{rank} [\lambda_i\mathbf{D}-A|B]=n\:\:\forall i\geqslant 1$, then for every $\varphi$ such that $(-\mathcal{M})^{k}(K_{*}\varphi)$ in $L^2(0,1)$ we have 
\begin{equation}\label{equ:prpo2.2}
	\int_{0}^{1}|\varphi(x,t)|^2 dx\leqslant R \int_{0}^{1}|(-\mathcal{M})^{k}(K_{*}\varphi)(x,t)|^2 dx
\end{equation}
for any $t\in [0,T)$ and any $k\geqslant (n-1)^2$, where $R$ only depends on $n$, $\mathbf{D}$ and $A$.	 
\end{proposition}
\begin{proof}  We adapt the same argument as in \cite{FerGonTer} to our degenerate case.
	Let denote by $K_i$ the matrices $\ds K_i=[\lambda_i\mathbf{D}-A|B]\in \mathcal{L}(\R^{nm};\R^n)\quad \forall i\geq 1$. Let $f\in L^2([0,1];\R^n)$ be  given
	\begin{equation}
	f=\sum_{i} a^i \mathbf{w}_i 
	\end{equation} 
	where $\ds a^i \in \R^n$ and $\ds a^i=0$ for all $i\geq p+1$ for some $p\geq 1$, then 
	$$\mathcal{K}_{*} f=\sum_{i} (K_i^{*}a^i) \mathbf{w}_i,\,\,\, (-\mathcal{M})^{k}\mathcal{K}_{*}f=\sum_{i} \lambda_i^{k}(K_i^{*}a^i) \mathbf{w}_i ,$$
	hence 
	\begin{equation}
	\|(-\mathcal{M})^{k}\mathcal{K}_{*}f \|_{L^2}^2=\sum_{i} \lambda_i^{2k}|K_i^{*}a^i|^2.
	\end{equation}

		Let us denote by $\eta_j^i$, for $1\leq j\leq n$, the real and nonnegative eigenvalues of $K_iK_i^*$. Then we have 
		\begin{equation}\label{equa:2.5}
		|K_i^*a^{i}|^2=(K_iK_i^*a^{i},a^{i})\geq \eta_1^i|a^{i}|^2
		\end{equation}
		There exists $c_1$ such that $$\mathrm{ det} K_iK_i^*\geqslant c_1\quad \forall i \geqslant 1.$$
		Indeed, let us set $p(\lambda)=\mathrm{ det} \tilde{K}(\lambda)\tilde{K}(\lambda)^*$  for all $\lambda$, with 
		$$\tilde{K}(\lambda):=[\lambda\mathbf{D}-A |B].$$
		Thus, $p(\lambda)$ is a polynomial function of degree  $2n(n-1)$, $p(\lambda)\geqslant 0$ for all $\lambda$ and $p(\lambda_i)\neq 0 $ for all $i$. Since the roots of $p(\lambda)=0$ are in a disk  of radius $R$ for some $R>0$, then, there exists $C_2>0$ such that $p(\lambda)\geqslant C_2$ for $|\lambda|\geqslant R$. Moreover, for some $\ell$, one has $\lambda_\ell>R$. Hence,
\begin{itemize}
	\setlength\itemsep{1em}
	\item Either $i\leqslant \ell-1$  and then $ \ds\mathrm{ det} K_iK_i^* \geqslant C_3 := \min_{ j\leq \ell-1}\mathrm{ det} K_j K_j^*.$
	\item Or $i\geqslant \ell$ and then $\lambda_i\geqslant \lambda_\ell> R $ and $\ds \mathrm{ det} K_iK_i^* \geqslant C_2 .$
\end{itemize}
Thus,  	$\mathrm{ det} K_iK_i^* \geqslant c_1=\min (C_2,C_3) .$	

Furthermore, for each $i\geqslant 1$  and each $\ell=1,\cdots,n$ there exists $\tilde{a}^{\ell} \in \R^n\setminus \{0\} $  such that
\begin{equation}
	\eta^i_\ell=\frac{(K_iK_i^*\tilde{a}^{\ell},\tilde{a}^{\ell} )}{|\tilde{a}^{\ell}|^2}\leqslant \|K_iK_i^*\|_2 \leqslant C_4(1+|\lambda_i|^{2(n-1)}),
\end{equation}
where $\|\cdot\|_2$ in the usual Euclidean norm in $\mathcal{L}(\R^n)$. 
Then we infer 
 $$\ds \eta^i_1=\frac{\mathrm{ det } K_iK_i^*}{\underset{\ell\geq 2}{\prod}\eta^i_\ell}\geq \frac{c_1}{  C_4^{n-1}(1+|\lambda_i|^{2(n-1)})^{n-1}}\geq C_5 |\lambda_i|^{-2(n-1)^2}$$ 

 Coming back to \eqref{equa:2.5} we get 
 $$|K_i^*a^i|^2=(K_i K_i^* a^i,a^i) \geqslant \eta_1^i |a^i|^2\geqslant C_5 |\lambda_i|^{-2(n-1)^2} |a^i|^2.$$
 Therefore 
 $$\|(-\mathcal{M})^k K_{*}f\|_{L^2}^2=\sum_{i}\lambda_i^{2k}|K_i^* a^i|^2\geqslant C_5 \sum_{i} |\lambda_i|^{2(k-(n-1)^2)} |a^i|^2\geqslant C \|f\|_{L^2}^2.$$
 
As this is true for all $f$ spanned by a finite amount of the $\mathbf{w}_i$, then we infer that this must also hold for all $f \in L^2((0,1);\R^n)$ such that $(-\mathcal{M})^k K_{*}f \in L^2((0,1);\R^n)$. In particular, we find \eqref{equ:prpo2.2}.
\end{proof}

From now on, we consider $\phi$  with the monomial derivative
$\mathcal{M}^i\partial_t^j \phi \in L^2(0,T;H_a^2(0,1))$  for every $i,j\in\N$, a solution of the following scalar degenerate parabolic equation of order $2n$ in space
\begin{equation}\label{scalar:parabolic}
\begin{cases}
P(\partial_t,\mathcal{M})\phi=0 & \text{in } Q,\\
\mathbf{C}\mathcal{M}^k \phi=0,  \quad  k\geqslant 0, & \text{on } \Sigma,
\end{cases}
\end{equation}
where $P(\partial_t,\mathcal{M})$ is the operator defined  by $\ds P(\partial_t,\mathcal{M})= det (\partial_t I_d+\mathbf{D}^* \mathcal{M}+A^*).$\\

Now we prove all components of every solution of the adjoint system \eqref{adjsyst24} are solutions of the scalar PDE \eqref{scalar:parabolic}.  So, it will be necessary to establish  Carleman  estimate for the scalar PDE \eqref{scalar:parabolic}.\\
First we recall The following result \cite{benfama,A4,FerGonTer}. 
\begin{proposition}\label{pro:3.3}
		Let $z_0\in \mathbb{D}^n$ and let $z=(z_1,\cdots,z_n)^*$ be the corresponding solution of problem \eqref{adjsyst24}. Then, $z \in \mathcal{C}^k([0,T];D(\mathcal{M}^p)^n)$ for every $k,\,p\geqslant 0$, and for every $i$,  $z_i$ solves equation \eqref{scalar:parabolic}. 
\end{proposition}

 The following proposition is the crucial result in this paper, since it generalize \cite[Lemma 4.1]{FerGonTer} to the case where the Jordan block size exceeds 4.

\begin{proposition}\label{propo3.4}
		for any $k\geq 0$ and $j\geq 0$, we can find an integer $m(k,j)\geq 0$, a constant $C(k,j)>0$ and an open set $\omega(k,j)$ satisfying $\omega\Subset \omega(k,j) \Subset \omega_1$, such that
	
	\begin{equation} \label{equ:propo3.4}
	I(\tau,(-\mathcal{M})^k\partial_t^j\phi) \leqslant C(k,j)\int\!\!\!\!\!\int_{\omega(k,j)\times(0,T)} (s\theta)^{m(k,j)} |\phi|^2 e^{2s\Phi}dxdt 
	\end{equation}
	where $\phi$ satisfies \eqref{scalar:parabolic} and 
	$$ \ds I(\tau,z)=\int\!\!\!\!\!\int_{Q}\Big((s\theta)^{\tau-1} z_t^2 +(s\theta)^{\tau-1} (\mathcal{M}z)^2+(s\theta)^{\tau+1}a(x)z_x^2
	+(s\theta)^{\tau+3}\frac{x^2}{ a(x)}z^2 \Big)e^{2s\varphi} dx dt $$
	
\end{proposition}

\begin{proof}
		We will prove \eqref{equ:propo3.4} by induction on $k$ and $j$ in two steps.\\
	{\it {\bf Step 1 : } Proof of \eqref{equ:propo3.4} for $k=j=0$}\\
	Let us see that, if s large enough, one has
	 	\begin{equation} \label{equ:propo3.400}
	 I(\tau,\phi) \leqslant C(0,0)\int\!\!\!\!\!\int_{\omega(0,0)\times(0,T)} (s\theta)^{m(0,0)} |\phi|^2 e^{2s\Phi}dxdt 
	 \end{equation}
for some $m(0,0),\,\, C(0,0)$ and  $\omega(0,0)$.\\
We assume that $\mathbf{D}$ satisfies the assumptions \eqref{cond1}-\eqref{jordan:block}, then we have for some $p\geq 1$	 
 	 
\begin{equation}
I_d\partial_t+\mathbf{D}^*\mathcal{M}+A^*= \begin{bmatrix} 
H_1(\partial_t,\mathcal{M}) & A^*_{2 1} &\cdots & A^*_{p 1} \\
A^*_{1 2}& H_2(\partial_t,\mathcal{M})&\cdots  & A^*_{p 2} \\
\vdots&\vdots & \ddots& \vdots\\
A^*_{1 p}&A^*_{2 p} &\cdots & H_p(\partial_t,\mathcal{M}) 
\end{bmatrix}
\end{equation}

where $H_i(\partial_t,\mathcal{M})$ is the non-scalar operator $H_i(\partial_t,\mathcal{M}):=I_d\partial_t+J_i^*\mathcal{M}+A_{i i}^*$,  the  $J_i^*$ are Jordan blocks,
i.e. each of them is of the form \eqref{jordan:block} for some $d_i\in \C$ and the $A_{i j}$ provide the corresponding block decomposition of $A$. Thus we can write \eqref{scalar:parabolic} as follow 
\begin{equation}\label{equ:3.10}
	\prod_{i=1}^{p}\mathrm{det}  H_i(\partial_t,\mathcal{M})\phi=F(\phi)
\end{equation}
in the term $F(\phi)$ we find the composition of at most $p-2$ operators of kind $\mathrm{det}  H_i(\partial_t,\mathcal{M})$ applied to $\phi$. Let us define  the functions $\psi_i$ by 
$$\psi_1=\phi,\,\psi_2=\mathrm{det}  H_1(\partial_t,\mathcal{M})\psi_1,\cdots,\psi_p=\mathrm{det}  H_{p-1}(\partial_t,\mathcal{M})\psi_{p-1}.$$
Thus the equation \eqref{equ:3.10} can be written as 
\begin{equation}\label{equ:3.11}
	\begin{cases}
	\mathrm{det}  H_{p}(\partial_t,\mathcal{M})\psi_{p}=F(\phi)\\
	\mathrm{det}  H_{p-1}(\partial_t,\mathcal{M})\psi_{p-1}=\psi_p\\
	\cdots\quad\,\cdots\quad\,\cdots \\
	\mathrm{det}  H_1(\partial_t,\mathcal{M})\psi_1=\psi_2,
	\end{cases}
\end{equation}
 by hypothesis, we  have 
\begin{equation*}
	\mathbf{C}\phi=\mathbf{C}\psi_2=\cdots=\mathbf{C}\psi_p=0 \,\, \text{   on }\,\,\Sigma
\end{equation*}
Let us consider the first PDE of \eqref{equ:3.11}, assume that $J_p$ is a Jordan block of dimension $r$ associated to the complex eigenvalue $\alpha$ with $\mathrm{Re }(\alpha)>0$  and let denote by $\eta_1,\cdots,\eta_r$ the diagonal components of $A_{p p}$. Then this PDE can be rewritten as 
\begin{equation}\label{equ:3.12}
	 	\prod_{i=1}^{r}  \big(\partial_t+\alpha\mathcal{M}+\eta_i\big)\psi_p=F(\phi)-G(\psi_p),
\end{equation}
where $G(\psi_p)$ is a linear combination of partial derivatives of $\psi_p$.\\
Again, let us introduce the new variables 
$$\zeta_{r}=\psi_p,\,\zeta_{r-1}=(\partial_t+\alpha\mathcal{M}+\eta_r)\zeta_{r},\cdots,\zeta_{1}=(\partial_t+\alpha\mathcal{M}+\eta_2)\zeta_{2}.$$

Therefore, we can rewrite \eqref{equ:3.11}  as a first-order system for the $\zeta_i$ :  
\begin{equation}\label{equ:3.13}
\begin{cases}
(\partial_t+\alpha\mathcal{M}+\eta_1)\zeta_{1}=F(\phi)-G(\psi_p)\\
(\partial_t+\alpha\mathcal{M}+\eta_2)\zeta_{2}=\zeta_{1}\\
\cdots\quad\,\cdots\quad\,\cdots \\
(\partial_t+\alpha\mathcal{M}+\eta_r)\zeta_{r}=\zeta_{r-1}.
\end{cases}
\end{equation}
with \begin{equation*}
\mathbf{C}\zeta_{1}=\mathbf{C}\zeta_{2}=\cdots=\mathbf{C}\zeta_{r}=0 \,\, \text{   on }\,\,\Sigma.
\end{equation*}
 
Notice that $|G(\psi_p)|^2$ is bounded by a sum of squares of derivatives of $\psi_p$. More precisely, we have $|G(\psi_p)|^2\leq C I_{G}(\psi_p)$, with 

    \begin{equation}\label{equ:3.17}
	I_{G}(\psi_p):= \sum_{\ell=0}^{r-1}\sum_{j_1,\cdots,j_{\ell}=0}^{r}
	\sum_{b=0}^{r-(\ell+1)}|(-\mathcal{M})^{b}
	\prod_{i=1}^{\ell}(\partial_t+\alpha\mathcal{M}+\eta_{j_i})\psi_p|^2
	\end{equation}

 The following $I_{\alpha,\mathbf{A}}(\tau,z)$ term already used in \cite{benfama} is defined by  
			\small{
				$$\ds I_{\alpha,\mathbf{A}}(\tau,z)=\int\!\!\!\!\!\int_{Q}\Big(s^{\tau-1}\theta^{\tau-1}\alpha^2 z_t^2 +s^{\tau-1}\theta^{\tau-1} \alpha^2 (\mathcal{M}z)^2 +s^{\tau+1}\theta^{\tau+1}a(x)\alpha^2 z_x^2
				+s^{\tau+3}\theta^{\tau+3}\frac{x^2}{ a(x)}\alpha^2 z^2 \Big)e^{2s\mathbf{A}} dx dt$$}
where $\mathbf{A}\in \{\varphi,\Phi\}$ and $\alpha\in\{\xi,\varsigma\}$ used in  Proposition \ref{propo3.6} and  Proposition \ref{proposition26}.\\
Applying Carleman estimates \eqref{eqproposition23} established in Proposition \ref{propo3.6} to the first PDE in \eqref{equ:3.13}, we get

\begin{equation*}
I_{\xi,\varphi}(\tau,\zeta_1)\leqslant C\Big(\int\!\!\!\!\!\int_{Q} \xi^2 (s\theta)^{\tau}(|F(\phi)|^2(t,x)+ I_{G}(\psi_p))e^{2s\varphi}dx dt+\int\!\!\!\!\!\int_{Q_{\omega'}} (s\theta)^{2+\tau} \zeta_1^2 e^{2s\varphi} dx dt\Big)
\end{equation*}

And for the $j^{\text{th}}$ PDE in \eqref{equ:3.13}  where $j=2,\cdots,r$, we have 

\begin{equation*}
I_{\xi,\varphi}(\tau+3(j-1),\zeta_j)\leqslant C\Big(\int\!\!\!\!\!\int_{Q} \xi^2 (s\theta)^{\tau+3(j-1)}|\zeta_{j-1}|^2 e^{2s\varphi}dx dt+\int\!\!\!\!\!\int_{Q_{\omega'}} (s\theta)^{\tau+3(j-1)+2} |\zeta_j|^2 e^{2s\varphi} dx dt\Big)
\end{equation*}
Consequently, an appropriate linear combination of the terms in the left hand sides absorbes the global
weighted integrals of $\ds |\zeta_j|^2$ for $j=2,\cdots,r$.

\noindent
\begin{multline*}
	 \ds \sum_{j=1}^{r} I_{\xi,\varphi}(\tau+3(j-1),\zeta_j)\leqslant C\Big( \sum_{j=1}^{r}\int\!\!\!\!\!\int_{Q_{\omega'}} (s\theta)^{\tau+3(j-1)+2} |\zeta_j|^2 e^{2s\varphi} dx dt \Big)\\
	 + C\Big(\int\!\!\!\!\!\int_{Q} \xi^2 (s\theta)^{\tau}(|F(\phi)|^2(t,x)+ I_{G}(\psi_p))e^{2s\varphi}dx dt \Big)\numberthis\label{equ:3.16}
\end{multline*}

The next task will be to add some extra terms on the left hand side of the previous inequality. To this end,
we reason as follow.  We apply $-\mathcal{M} $ to the $j^{\text{th}}$ PDE in \eqref{equ:3.13}  where $j=2,\cdots,r$, we have 
$$ -(\partial_t+\alpha\mathcal{M}+\eta_j)\mathcal{M}\zeta_{j}=-\mathcal{M}\zeta_{j-1}. $$
\vspace*{.5cm}

Using Proposition \ref{propo3.6}, we get 
\begin{align*}
\ds I_{\xi,\varphi}(\tau+3(j-2)-1,\mathcal{M}\zeta_{j})	&\leqslant C\Big(\int\!\!\!\!\!\int_{Q} \xi^2 (s\theta)^{\tau+3(j-2)-1}(\mathcal{M}\zeta_{j-1})^2e^{2s\varphi(t,x)}dx dt\\
& \qquad +\int\!\!\!\!\!\int_{Q_{\omega'}} (s\theta)^{\tau+3(j-2)+2} (\mathcal{M}\zeta_{j})^2 e^{2s\varphi(t,x)}dxdt\Big) \\
&\leqslant C I_{\xi,\varphi}(\tau+3(j-2),\zeta_{j-1})+  C I_{\xi,\varphi}(\tau++3(j-1),\zeta_j).
\end{align*}

Then, we can add all these new terms to the left hand side of \eqref{equ:3.16} and, for a new positive constant C, obtain

\begin{multline*}
\ds \sum_{j=1}^{r} I_{\xi,\varphi}(\tau+3(j-1),\zeta_j)+ \sum_{j=2}^{r} I_{\xi,\varphi}(\tau+3(j-2)-1,\mathcal{M}\zeta_j)\\
\leqslant C\Big( \sum_{j=1}^{r}\int\!\!\!\!\!\int_{Q_{\omega'}} (s\theta)^{\tau+3(j-1)+2} |\zeta_j|^2 e^{2s\varphi} dx dt \Big)
+ C\Big(\int\!\!\!\!\!\int_{Q} \xi^2 (s\theta)^{\tau}(|F(\phi)|^2(t,x)+ I_{G}(\psi_p))e^{2s\varphi}dx dt \Big)\numberthis\label{equ:3.20bis}
\end{multline*}

We can continue the previous process and add better global terms on the left hand side of \eqref{equ:3.20bis}. Thus, if we apply $(-\mathcal{M})^2$ to  the $j^{\text{th}}$ PDE in \eqref{equ:3.13}  where $j=k+1,\cdots,r$ and $k=2,\cdots,r-1$, we have  
$$(\partial_t+\alpha\mathcal{M}+\eta_j)(-\mathcal{M})^k\zeta_{j}=(-\mathcal{M})^k\zeta_{j-1}$$
 we use again Proposition \ref{propo3.6}  for the previous equations 
\begin{align*}
I_{\xi,\varphi}(\tau+3(j-(k+1))-k,(-\mathcal{M})^k\zeta_{j})	&\leqslant C\Big(\int\!\!\!\!\!\int_{Q} \xi^2 (s\theta)^{\tau+3(j-(k+1))-k}|(-\mathcal{M})^k\zeta_{j-1}|^2 e^{2s\varphi(t,x)}dx dt\\
&\qquad +\int\!\!\!\!\!\int_{Q_{\omega'}} (s\theta)^{\tau+3(j-(k+1))-k+2} |(-\mathcal{M})^k\zeta_{j}|^2 e^{2s\varphi(t,x)}dxdt\Big) \\
&\leqslant C I_{\xi,\varphi}(\tau+3(j-(k+1))-k+1,(-\mathcal{M})^{k-1}\zeta_{j-1})\\
&\qquad+  C I_{\xi,\varphi}(\tau+3(j-(k+1))-k+2,(-\mathcal{M})^{k-1}\zeta_{j})\\
&\leqslant C I_{\xi,\varphi}(\tau+3((j-1)-k)-(k-1),(-\mathcal{M})^{k-1}\zeta_{j-1})\\
&\qquad+  C I_{\xi,\varphi}(\tau+3(j-k)-(k-1),(-\mathcal{M})^{k-1}\zeta_{j}).\numberthis\label{equa:3.20}
\end{align*}

 Let us denote by $\ds \mathcal{J}_{\xi,\varphi}(\tau,\zeta)$ the following  sum 
\begin{equation}
\ds \mathcal{J}_{\xi,\varphi}(\tau,\zeta) =\sum_{k=0}^{r-1}\sum_{j=k+1}^{r}I_{\xi,\varphi}(\tau+3(j-(k+1))-k,(-\mathcal{M})^k\zeta_{j})	 
\end{equation}
where $\zeta=(\zeta_1,\cdots,\zeta_r)$.
Then, we can add all these new terms \eqref{equa:3.20} to the left hand side of \eqref{equ:3.20bis} and, for a new positive constant C, obtain

\begin{multline*}
\ds \mathcal{J}_{\xi,\varphi}(\tau,\zeta)
\leqslant C\Big( \sum_{j=1}^{r}\int\!\!\!\!\!\int_{Q_{\omega'}} (s\theta)^{\tau+3(j-1)+2} |\zeta_j|^2 e^{2s\varphi} dx dt \Big)
+ C\Big(\int\!\!\!\!\!\int_{Q} \xi^2 (s\theta)^{\tau}(|F(\phi)|^2(t,x)+ I_{G}(\psi_p))e^{2s\varphi}dx dt \Big)\numberthis\label{equa:3.22}
\end{multline*}
with $s$ sufficiently large.\\
From now on, we fix $s$ sufficiently large and we  try to replace the local terms in \eqref{equa:3.22} corresponding to $\zeta_1,\cdots,\zeta_{r-1}$ by a term of the form ($\psi_p=\zeta_r$) using the    same computation \cite[Lemma 3.7]{FadiliManiar}, we can show the existence of a constant $C>0$ and an integer $\ell_1$ such that :
{\small
	\begin{equation}\label{equ:3.21}
	\ds \mathcal{J}_{\xi,\varphi}(\tau,\zeta) \leqslant C \int\!\!\!\!\!\int_{Q_{\omega'}} s^{\ell_1}\theta^{\ell_1} \psi_p^2 e^{2s\varphi}dxdt  
	+ C\int\!\!\!\!\!\int_{Q} \xi^2s^\tau \theta^{\tau}(|F(\phi)|^2+ I_{G}(\psi_p))e^{2s\varphi}dx dt. 
	\end{equation}
}

Since the operators $(\partial_t+\alpha\mathcal{M}+\eta_j),\,\, j=1,\cdots,r$  commute, we see that \eqref{equ:3.12}  can be rewritten equivalently in the form
\begin{equation*}
\prod_{i=1}^{r}  \big(\partial_t+\alpha\mathcal{M}+\eta_{\sigma(i)}\big)\psi_p=F(\phi)-G(\psi_p),
\end{equation*}

where $\sigma$ is any permutation in $\mathcal{P}_r$. Hence, we can introduce the new variables
$$\zeta_{r}^{\sigma}=\psi_p,\,\zeta_{r-1}^{\sigma}=(\partial_t+\alpha\mathcal{M}+\eta_{\sigma(r)})\zeta_{r}^{\sigma},\cdots,\zeta_{1}^{\sigma}=(\partial_t+\alpha\mathcal{M}+\eta_{\sigma(2)})\zeta_{2}^{\sigma}.$$

and we can also rewrite \eqref{equ:3.12}  as a first-order system for the $\zeta_i^{\sigma}$ : 

\begin{equation}\label{equ:3.13bis}
\begin{cases}
(\partial_t+\alpha\mathcal{M}+\eta_{\sigma(1)})\zeta_{1}^{\sigma}=F(\phi)-G(\psi_p)\\
(\partial_t+\alpha\mathcal{M}+\eta_{\sigma(2)})\zeta_{2}^{\sigma}=\zeta_{1}^{\sigma}\\
\cdots\quad\,\cdots\quad\,\cdots \\
(\partial_t+\alpha\mathcal{M}+\eta_{\sigma(r)})\zeta_{r}^{\sigma}=\zeta_{r-1}^{\sigma}.
\end{cases}
\end{equation}

Again, with \begin{equation*}
\mathbf{C}\zeta_{1}^{\sigma}=\mathbf{C}\zeta_{2}^{\sigma}=\cdots=\mathbf{C}\zeta_{r}^{\sigma}=0 \,\, \text{   on }\,\,\Sigma.
\end{equation*}

As before, we obtain an estimate like \eqref{equ:3.21} 
{\small
	\begin{equation*}
	\ds \mathcal{J}_{\xi,\varphi}(\tau,\zeta^{\sigma}) \leqslant C \int\!\!\!\!\!\int_{Q_{\omega'}} s^{\ell_1}\theta^{\ell_1} \psi_p^2 e^{2s\varphi}dxdt  
	+ C\int\!\!\!\!\!\int_{Q} \xi^2s^\tau \theta^{\tau}(|F(\phi)|^2+ I_{G}(\psi_p))e^{2s\varphi}dx dt. 
	\end{equation*}
}
In this inequality we have used the notation $\zeta^{\sigma}=(\zeta^{\sigma}_1,\cdots,\zeta^{\sigma}_r) $.
Now, let us define  $\ds \mathcal{I}_{\xi,\varphi}(\tau,\zeta)$ by 
$$\ds \mathcal{I}_{\xi,\varphi}(\tau,\psi_p)=\sum_{\sigma\in\mathcal{P}_r}\ds \mathcal{J}_{\xi,\varphi}(\tau,\zeta^{\sigma})$$
we have

{\small
	\begin{equation*}
	\ds \mathcal{I}_{\xi,\varphi}(\tau,\psi_p) \leqslant C \int\!\!\!\!\!\int_{Q_{\omega'}} s^{\ell_1}\theta^{\ell_1} \psi_p^2 e^{2s\varphi}dxdt  
	+ C\int\!\!\!\!\!\int_{Q} \xi^2s^\tau \theta^{\tau}(|F(\phi)|^2+ I_{G}(\psi_p))e^{2s\varphi}dx dt. 
	\end{equation*}
}

Observe that all the terms in $\ds I_{G}(\psi_p)$ are also in the left multiplied by weights of the form $\ds s^{\kappa} \theta^{\kappa}e^{2s\varphi}$ with $\kappa>0$. Consequently, for sufficiently large $s$, these terms are absorbed and we find  
{\small
	\begin{equation}\label{equ:3.25}
	\ds \mathcal{I}_{\xi,\varphi}(\tau,\psi_p) \leqslant C \int\!\!\!\!\!\int_{Q_{\omega'}} s^{\ell_1}\theta^{\ell_1} \psi_p^2 e^{2s\varphi}dxdt  
	+ C\int\!\!\!\!\!\int_{Q} \xi^2s^\tau \theta^{\tau}|F(\phi)|^2e^{2s\varphi}dx dt. 
	\end{equation}
}

Let us now consider the second PDE in \eqref{equ:3.11}.  Arguing in the same way, we deduce
the following estimate for $\psi_{p-1}$ :

{\small
	\begin{equation}\label{equ:3.27}
	\ds \mathcal{I}_{\xi,\varphi}(\tau,\psi_{p-1}) \leqslant C \int\!\!\!\!\!\int_{Q_{\omega'}} s^{\ell_2}\theta^{\ell_2} |\psi_{p-1}|^2 e^{2s\varphi}dxdt  
	+ C\int\!\!\!\!\!\int_{Q} \xi^2s^\tau \theta^{\tau}|\psi_{p}|^2e^{2s\varphi}dx dt. 
	\end{equation}
}

The corresponding similar estimate also holds for $\psi_{p-1}$, {\it etc}. Thus, after addition and taking into account  that $\psi_{1}=\phi$ and the global integrals of $\psi_{p},\cdots,\psi_{2}$ in the right hand side are smaller than the terms in the left, we get an estimate for all the $\psi_{i}$ :

{\small
	\begin{multline}\label{equ:3.28}
	\sum_{i=1}^{p} \mathcal{I}_{\xi,\varphi}(\tau,\psi_{i}) \leqslant C \Big(\int\!\!\!\!\!\int_{Q} \xi^2s^\tau \theta^{\tau}|F(\phi)|^2e^{2s\varphi}dx dt \\
	+\int\!\!\!\!\!\int_{Q_{\omega'}} (s\theta)^{\ell_2} |\phi|^2 e^{2s\varphi}dxdt  
	+ \sum_{i=2}^{p}\int\!\!\!\!\!\int_{Q_{\omega'}} (s\theta)^{\ell_2}|\psi_{i}|^2e^{2s\varphi}dx dt\Big). 
	\end{multline}
}

Again, using the cascade structure of system \eqref{equ:3.11}, all the local integrals in the right can be absorbed by the left hand side, with the exception of the local weighted integral of $|\phi|^2$.  All we have to do is to enlarge the open set ω1 and argue like in the passage from \eqref{equa:3.22} to \eqref{equ:3.21}. Therefore, the following is obtained:

{\small
	\begin{equation*}\label{equ:3.29}
	\sum_{i=1}^{p} \mathcal{I}_{\xi,\varphi}(\tau,\psi_{i}) \leqslant C \Big(\int\!\!\!\!\!\int_{Q} \xi^2s^\tau \theta^{\tau}|F(\phi)|^2e^{2s\varphi}dx dt 
	+\int\!\!\!\!\!\int_{Q_{\omega'}} (s\theta)^{\ell_2} |\phi|^2 e^{2s\varphi}dxdt  
	\Big). 
	\end{equation*}
}

 we see that, taking into account that the operators $\mathrm{det}  H_i(\partial_t,\mathcal{M})$, commute for $i=1,\cdots,p$. Then we can rewrite  \eqref{equ:3.10} in the form
  
 \begin{equation*}
 \prod_{i=1}^{p}\mathrm{det}  H_{\sigma(i)}(\partial_t,\mathcal{M})\phi=F(\phi)
 \end{equation*}
  where $\sigma$  is any permutation in $\mathcal{P}_n$. Then we have the following  equivalent formulation of \eqref{scalar:parabolic}   
 \begin{equation*}
 \begin{cases}
 \mathrm{det}  H_{\sigma(p)}(\partial_t,\mathcal{M})\psi_{p}^{\sigma}=F(\phi)\\
 \mathrm{det}  H_{\sigma(p-1)}(\partial_t,\mathcal{M})\psi_{p-1}^{\sigma}=\psi_p^{\sigma}\\
 \cdots\quad\,\cdots\quad\,\cdots \\
 \mathrm{det}  H_{\sigma(1)}(\partial_t,\mathcal{M})\psi_1^{\sigma}=\psi_2^{\sigma}
 \end{cases}
 \end{equation*}
thus, we can also get an estimate of the same form where, now, we have in the left global weighted integrals of
$\phi,\,\psi_{2}^{\sigma},\cdots,\psi_{p}^{\sigma}$. Taking into account that $ F(\phi)$ is a sum of terms where, at most, $p-2$  operators of the kind $\mathrm{det}  H_{j}(\partial_t,\mathcal{M}) $ are applied to $\phi$. Since $\sigma$  is arbitrary in $\mathcal{P}_n$, using all these possible estimates together and arguing as above, it becomes also clear that the terms containing $ |F(\phi)|^2$ can be controlled by the terms in the left. This gives 

{\small
	\begin{equation}\label{equ:3.29}
	\sum_{\sigma\in\mathcal{P}_n}\sum_{i=1}^{p} \mathcal{I}_{\xi,\varphi}(\tau,\psi_{i}^{\sigma}) \leqslant C \int\!\!\!\!\!\int_{Q_{\omega'}} (s\theta)^{\ell_2} |\phi|^2 e^{2s\varphi}dxdt.  
\end{equation}
}
 Likewise,  by applying Proposition \ref{proposition26} we infer
{\small
	\begin{equation}\label{equ:3.30}
	\sum_{\sigma\in\mathcal{P}_n}\sum_{i=1}^{p} \mathcal{I}_{\varsigma,\Phi}(\tau,\psi_{i}^{\sigma}) \leqslant C \int\!\!\!\!\!\int_{Q_{\omega'}} (s\theta)^{\ell_3} |\phi|^2 e^{2s\Phi}dxdt.  
	\end{equation}
}
From \eqref{equ:3.29} and \eqref{equ:3.30} we deduce 
{\small
	\begin{equation}\label{equ:3.30}
	\sum_{\sigma\in\mathcal{P}_n}\sum_{i=1}^{p} \mathcal{I}(\tau,\psi_{i}^{\sigma}) \leqslant C \int\!\!\!\!\!\int_{Q_{\omega'}} (s\theta)^{\tilde{\ell}} |\phi|^2 e^{2s\Phi}dxdt.  
	\end{equation}
}
where $\tilde{\ell}=\max(\ell_2,\ell_3)$. This proves \eqref{equ:propo3.400}.\\
{\it {\bf Step 2 : } Induction on $k$ and $j$. }\\
Let us now assume that \eqref{equ:propo3.4} is true for any $k'=0,1,\cdots,k$  any $j'=0,1,\cdots,j$ and any solution to \eqref{scalar:parabolic}  and let us prove \eqref{equ:propo3.4}  (for instance) with $k$ replaced by $k+1$; the proof with the same $k$ and $j$ replaced by $j+1$ is essentially the same.

Since $\hat{\phi}:=(-\mathcal{M})\phi$  also satisfies \eqref{scalar:parabolic}, we have by hypothesis

\begin{align*} 
I(\tau,(-\mathcal{M})^{k+1}\partial_t^j\phi) & = I(\tau,(-\mathcal{M})^k\partial_t^j\hat{\phi})\\
&  \leqslant C(k,j)\int\!\!\!\!\!\int_{\omega(k,j)\times(0,T)} (s\theta)^{m(k,j)} |\hat{\phi}|^2 e^{2s\Phi}dxdt\\
&\leqslant C(k,j)\int\!\!\!\!\!\int_{\omega(k,j)\times(0,T)} (s\theta)^{m(k,j)} |\mathcal{M}\phi|^2 e^{2s\Phi}dxdt\\
&\leqslant C(k,j)C'\int\!\!\!\!\!\int_{\omega(k,j)\times(0,T)} (s\theta)^{m(k,j)} |\mathcal{M}\phi|^2 e^{2s\varphi}dxdt\\
& \leqslant C(k,j)C' I(m(k,j)+1,\phi)\\
&\leqslant C'' \int\!\!\!\!\!\int_{\mathcal{O}'\times(0,T)} (s\theta)^{m'} |\phi|^2 e^{2s\Phi}dxdt.
\end{align*}
Thus,  that there exist $m(k+ 1,j)$, $C(k+ 1,j)$ and $\omega(k+1,j)$ such that
\begin{align*} 
I(\tau,(-\mathcal{M})^{k+1}\partial_t^j\phi)\leqslant C(k+1,j)\int\!\!\!\!\!\int_{\omega(k+1,j)\times(0,T)} (s\theta)^{m(k+1,j)} |\mathcal{M}\phi|^2 e^{2s\Phi}dxdt
\end{align*}
This ends the proof.
\end{proof}


\section{Proof of the main result}
\begin{proof}[Proof of  Theorem \ref{main result}]
	Let us first assume that \eqref{syst24}  is null-controllable. If we have $\mathrm{rank}[\lambda_i\mathbf{D}+A|B]\leq n-1 $ for some $i$, then the associated ordinary differential system is not null-controllable. This means there exists $z_T\in\R^n\setminus\{0\}$ such that the solution to the Cauchy problem
	 \begin{equation}
	\left\lbrace \begin{array}{lll}
	-\partial_t y+(\lambda_i\mathbf{D}^* +A^*) y =0 & in & Q, \\
	y(T)=y_T   & on & (0,T), 
	\end{array}
	\right.
	\end{equation} 
	satisfies $$B^*y(t)=0.$$
		If we now set $\phi_T=y_T \mathbf{w}_i$, where $\mathbf{w}_i$  is an eigenfunction associated to $\lambda_i$, we see that the corresponding solution to the adjoint system \eqref{adjsyst24} cannot satisfy the observability inequality \eqref{observ:inequality}.
Consequently, \eqref{condiequivalente} must hold.\\	
Conversely, let us assume that  \eqref{condiequivalente} is satisfied, and let us prove that the system \eqref{syst24} is null-controllable.
 
Let $z$  be the solution to the adjoint system \eqref{adjsyst24} corresponding to a final data $z_T$,
by Proposition \ref{prpo2.2}, for $k\geqslant (n-1)^2$ there exists a positive constant $C$ depends on $n$, $\mathbf{D}$ and $A$ such that 
\begin{equation}\label{equ:4.2}
\int_{0}^{1}|z(x,t)|^2 dx\leqslant C \int_{0}^{1}|(-\mathcal{M})^{k}(K_{*}z)(x,t)|^2 dx
\end{equation}
for any $t\in [0,T)$. From \eqref{operteurdekalman} the components of $K_{*}z$ are appropriate linear combinations of the components of $z$ and their second-order in space derivatives.
Notice again that, for all $t\in [0,T)$,  $z(\cdot,t)$  is regular enough to give a sense to $(-\mathcal{M})^{k}(K_{*}z)$, which belongs to $L^2(0,1)$.

By Proposition \ref{pro:3.3}  $z \in \mathcal{C}^k([0,T];D(\mathcal{M}^p)^n)$ for every $k,\,p\geqslant 0$ and for every $i$  the component $z_i$ of $z$ solves equation \eqref{scalar:parabolic}. Thus, we can write \eqref{equ:propo3.4}  for any component of $B^*z$. This gives the following inequality for all $j,k\geq 0$ and all $ \ell=1,\cdots,m$

\begin{align*} 
\int\!\!\!\!\!\int_{Q} |(B^*((-\mathcal{M})^k\partial_t^j z))_{\ell}|^2 e^{2s\varphi}dxdt
&=\int\!\!\!\!\!\int_{Q}  |(-\mathcal{M})^k\partial_t^j(B^* z)_{\ell}|^2 e^{2s\varphi}dxdt\\
&\leqslant C(k,j)\int\!\!\!\!\!\int_{\omega(k,j)\times(0,T)}  |(B^* z)_{\ell}|^2 e^{2s\Phi}dxdt 
\end{align*}
let 
$\ds M_0=\underset{x\in (0,1)}{\max}\psi(x) $, thus 
\begin{align*} 
\int\!\!\!\!\!\int_{Q} |(-\mathcal{M})^k\mathcal{K}_{*} z|^2 e^{-2s\theta M_0}dxdt
&\leqslant C\sum_{\ell=1}^{m}\int\!\!\!\!\!\int_{Q} |(B^*((-\mathcal{M})^k\partial_t^j z))_{\ell}|^2 e^{2s\varphi}dxdt\\
&\leqslant C\sum_{\ell=1}^{m} C(k,j)\int\!\!\!\!\!\int_{\omega\times(0,T)}  |(B^* z)_{\ell}|^2 e^{2s\Phi}dxdt\\
&\leqslant C \int\!\!\!\!\!\int_{\omega\times(0,T)}  |B^* z|^2 e^{2s\Phi}dxdt \numberthis\label{equ:4.3} 
\end{align*}

From \eqref{equ:4.2}  and \eqref{equ:4.3}, we will easily deduce \eqref{observ:inequality} and, therefore, the null controllability of \eqref{syst24}.
 
\end{proof}
\section{Null controllability for semilinear systems}

Now we consider the following  semi-linear non-diagonalizable parabolic degenerate systems.
\begin{equation}\label{semilinsyst}
\left\lbrace \begin{array}{lll}
\partial_t Y=\mathbf{D}\mathcal{M}Y + F(Y)+B v \mathbbm{1}_{\omega}& in & Q, \\
\mathbf{C} Y=0   & on & \Sigma, \\
Y(0,x)=Y_0(x)\,\,  & in & (0,1),
\end{array}
\right.
\end{equation}
 where $F$ is a globally Lipschitz function depending only on $Y$ and $F(0)=0$. Our goal is to prove the null controllability of the system \eqref{semilinsyst}.

 We will use a standard strategy, as in  \cite{zua,A5,CanFrag,bahm}, which consists in using the linearization
 technique, the approximate null controllability, the variational approach and the Schauder
 fixed point theorem.


The system \eqref{semilinsyst} can be written as follow  

\begin{equation*}
\left\lbrace \begin{array}{lll}
\partial_t Y=(\mathbf{D}\mathcal{M}   +A_{Y}) Y + B v \mathbbm{1}_{\omega}& in & Q, \\
\mathbf{C} Y=0   & on & \Sigma, \\
Y(0,x)=Y_0(x)\,\,  & in & (0,1),
\end{array}
\right.
\end{equation*}
		 
where $A_{Y}$ is the matrix defined by 
\begin{equation*}
	 a_{i,j}^{Y}=\int_{0}^{1}\partial_j F_i(\tau Y)d\tau
\end{equation*}

We assume the following 
		
	\begin{equation}\label{condionalgsemi}
	\begin{cases}
	 F\in \mathcal{C}^1(\R^n)\\
	\mathrm{rank} [\lambda_i\mathbf{D}-A_{Y}|B]=n\:\:\forall i\geqslant 1\,\text{ For all } Y\in L^2(0,1)^n \,
	\end{cases}	
	\end{equation}

 Let  us recall  the set $X_T$ (see Theorem \ref{theo2.2}) induced with the norm 
 \begin{equation*}
 \|Y\|_{X_T}^2=\underset{t\in[0,T]}{\sup}\|Y(t)\|_{L^2(0,1)^n}^2 +\int_0^T\|\sqrt{a}Y(t)\|_{L^2(0,1)^n}^2 dt
 \end{equation*}

For a fixed $\widetilde{Y}$ in $X_T$, consider the associated linear system

\begin{equation}\label{linearized}
\left\lbrace \begin{array}{lll}
\partial_t Y=(\mathbf{D}\mathcal{M}   +A_{\widetilde{Y}}) Y + B v \mathbbm{1}_{\omega}& in & Q, \\
\mathbf{C} Y=0   & on & \Sigma, \\
Y(0,x)=Y_0(x)\,\,  & in & (0,1),
\end{array}
\right.
\end{equation}

and its adjoint system

	\begin{equation}\label{adjointlinearized}
	\left\lbrace \begin{array}{lll}
	\partial_t Z=(\mathbf{D}^*\mathcal{M}   +A_{\widetilde{Y}}^*) Z & in & Q, \\
	\mathbf{C} Z=0   & on & \Sigma, \\
	Z(0,x)=Z_0(x)\,\,  & in & (0,1),
	\end{array}
	\right.
	\end{equation}

Thus, from \eqref{condionalgsemi}, it follows that the matrix $A_{\widetilde{Y}}$ satisfies the algebraic condition \eqref{condiequivalente}.



 In order to construct a suitable fixed point operator, we start at first  by proving the uniqueness of the control with minimal norm. For a given  $\varepsilon>0$ and $Y_0\in L^2(0,A)^n$ we consider the following functional

\begin{equation}
	J_{\varepsilon,\widetilde{Y}}(v)=\frac12\int_{0}^{T}\| v\|_{L^2(0,1)^m}^2 dt+\frac{1}{2\varepsilon}\|Y(T)\|_{L^2(0,1)^n}^2
\end{equation}

\begin{equation}\label{equ:6.5}
J^*_{\varepsilon,\widetilde{Y}}(Z_0)=\frac12\iint_{Q_{\omega}}|B^* Z |^2 +\frac{\varepsilon}{2}\|Z_0\|_{L^2(0,1)^n}^2+\int_{0}^{1}\big(\sum_{i=1}^{n}Z_i(T)Y_{i0}\big)dx
\end{equation}
where $Y$ is the solution of \eqref{linearized} with initial data $Y_0$ and $Z$ is the solution of \eqref{adjointlinearized} with initial data $Z_0$. By a classical arguments, minimization problems 
\begin{equation*}
	\min\{ J_{\varepsilon,\widetilde{Y}}(v), Bv \in L^2(Q)^n  \} \,\, \text{ and } \,\, \min\{ J^*_{\varepsilon,\widetilde{Y}}(Z_0), Z_0 \in L^2(0,1)^n  \}  
\end{equation*}
admit unique solutions $v^{\varepsilon,\widetilde{Y}}$   and  $Z_0^{\varepsilon,\widetilde{Y}}$  such that 
 \begin{equation}\label{equ:6.6}
 	 v=B^*Z \mathbbm{1}_{\omega}
 \end{equation}
 
  \begin{equation}\label{equ:6.7}
  Z_0^{\varepsilon,\widetilde{Y}}=-\frac{1}{\varepsilon} Y^{\varepsilon,\widetilde{Y}}(T)
   \end{equation}
 where $Y^{\varepsilon,\widetilde{Y}}$ is the solution of \eqref{linearized}  associated to the control $v$ and $ Z^{\varepsilon,\widetilde{Y}}$ is the solution of the adjoint problem \eqref{adjointlinearized} with the initial data $ Z_0^{\varepsilon,\widetilde{Y}}$.  Since $J^*_{\varepsilon,\widetilde{Y}}(Z_0^{\varepsilon,\widetilde{Y}})\leqslant 0$ ,then from \eqref{equ:6.5}  and \eqref{equ:6.7} we infer 
 
  \begin{equation}\label{equ:6.8}
  \frac12\int\!\!\!\!\!\int_{Q_{\omega}}|B^* Z |^2
  +\frac{1}{2\varepsilon}\|Y^{\varepsilon,\widetilde{Y}}(T)\|_{L^2(0,1)^n}^2 \leqslant \|Z(T,\cdot) \|_{L^2(0,1)^n} \|Y(0,\cdot) \|_{L^2(0,1)^n} 
 \end{equation}
 
On the other hand, by the observability inequality \eqref{observ:inequality}

\begin{equation}\label{equ:6.9}
\|Z(T,\cdot) \|^2_{L^2(0,1)^n}\leqslant C \int\!\!\!\!\!\int_{\omega\times(0,T)} |B^*Z(t,x)|^2.
\end{equation} 

From the estimates \eqref{equ:6.6} \eqref{equ:6.8} and \eqref{equ:6.9}  we infer 
 
 \begin{equation}\label{equ:6.10}
  \frac12\int_{0}^{T}\| v^{\varepsilon,\widetilde{Y}} \|_{L^2(0,1)^m}^2 dt
 +\frac{1}{2\varepsilon}\|Y^{\varepsilon,\widetilde{Y}}(T)\|_{L^2(0,1)^n}^2 \leqslant C \|Y(0,\cdot) \|_{L^2(0,1)^n} 
 \end{equation}

 The uniqueness of the control $v^{\varepsilon,\widetilde{Y}}$  allows to define the operator
 \begin{align*}
 	 K_{\varepsilon}\,\,:\,\,& X_T \to X_T\\
 	                    & \,\, \widetilde{Y} \mapsto Y^{\varepsilon,\widetilde{Y}}.\numberthis\label{equa:6.11}
 \end{align*}
 Any fixed point $Y^{\varepsilon}$ of $K_{\varepsilon}$ is a solution of the semilinear system
\eqref{semilinsyst}  associated to $v^{\varepsilon,Y^{\varepsilon}}$  and it satisfies 
\begin{equation}
	 \|Y^{\varepsilon}\|^2_{(L^2(0,1))^n}\leqslant \varepsilon C.
\end{equation}
Indeed, suppose first that $Y_0 \in (H^1_a)^n$. From \eqref{therem22equ:2} and since the matrix $A$ is constant, we have the following estimate

\begin{multline*}
\sup_{t\in[0,T]}\|Y^{\varepsilon,\widetilde{Y}}(t)\|^2_{\left(H^1_{a}\right)^n} + \int_0^T\left( \| \partial_t Y^{\varepsilon,\widetilde{Y}}\|^2_{L^2}+\|\mathcal{M} Y^{\varepsilon,\widetilde{Y}}\|^2_{L^2} \right)dt\\
\leq C_T\left(\| Y_0^{\varepsilon,\widetilde{Y}}  \|^2_{\left( H^1_{a}\right)^n}+ \|v^{\varepsilon,\widetilde{Y}}\|^2_{(L^2(Q))^m}
\right)\numberthis
\end{multline*}  

Thus, by \eqref{equ:6.10} we infer 
\begin{align}
	\|Y^{\varepsilon,\widetilde{Y}} \|_{X_T} &\leqslant C \|Y_0 \|_{(H_a^1)^n},\label{equ:6.14} \\
	\|Y^{\varepsilon,\widetilde{Y}} \|_{Y_T} &\leqslant C \|Y_0 \|_{(H_a^1)^n}\label{equ:6.15}
\end{align}
where $Y_T:=H^1 \left(0,T;\left(L^2(0,1)\right)^n\right)\cap L^2\left(0,T;{\left( H_{a}^1\right)}^n \right)$  with the norm 
\begin{equation*}
\|Y \|_{Y_T}=\int_0^T\left(\|Y(t)\|^2_{(H^1_{a})^n} +  \| \partial_t Y \|^2_{L^2}+\|\mathcal{M} Y\|^2_{L^2} \right)dt.
\end{equation*}  

Thus, the range of $K_{\varepsilon}$ is include in the ball  $B(0,R)$ of $X_T$ with the radius 
$R= C \|Y_0 \|_{(H_a^1)^n}$ where $C$ is the constant used in \eqref{equ:6.14}. 
Then $K_{\varepsilon}(B(0,R))\subset B(0,R)$. Now let us prove that the operator $K_{\varepsilon}$ is continuous and compact. 
The compactness of $K_{\varepsilon}$ results from the compactness of the embedding 
\begin{equation}\label{equa:6.16}
Y_T\hookrightarrow  X_T 
\end{equation}
see \cite[Theorem 4.4]{CanFrag}.   For  the continuity, let us consider the sequence $\widetilde{Y}_n$ that converge to $\overline{\widetilde{Y}}$ in $ X_T$. To simplify, let denote $Y^{\varepsilon,\widetilde{Y}_n}$ and $v^{\varepsilon,\widetilde{Y}_n}$ respectively by $Y_n$
 and  $v_n$ (for a fixed $\varepsilon$). From \eqref{equ:6.15} we deduce that the sequence  $Y_n$is bounded in the space $ Y_T$. Thus we can extract a subsequence that converges weakly in $ Y_T$ to $\overline{Y}$ and strongly in $ X_T$ by dint of \eqref{equa:6.16}. Likewise, thanks to \eqref{equ:6.10} we can assume that $v_n$ converges weakly to $\overline{v}$. So $\overline{Y}$ is then the solution of \eqref{linearized} associated to $\overline{\widetilde{Y}}$ and $\overline{v}$. Therefore, in order to show that $K_{\varepsilon}(\overline{\widetilde{Y}})=\overline{Y}$ it suffices to prove that $\overline{v}=v^{\varepsilon,\overline{Y} }$. 
 From the definition of $v_n$, we have for all $v$ in $L^2(Q)^m$
 \begin{equation}\label{equ:6.17}
 \frac12\int_{0}^{T}\| v_n\|_{L^2(0,1)^m}^2 dt+\frac{1}{2\varepsilon}\|Y_n(T)\|_{L^2(0,1)^n}^2  \leqslant \frac12\int_{0}^{T}\| v\|_{L^2(0,1)^m}^2 dt+\frac{1}{2\varepsilon}\|Y^{\widetilde{Y}_n,v}(T)\|_{L^2(0,1)^n}^2,
 \end{equation}
where $Y^{\widetilde{Y}_n,v}$ is the solution of \eqref{linearized} associated to $\widetilde{Y}_n$ and $v$. Passing to the limit in the inequality \eqref{equ:6.17}, one has for all $v$ in $L^2(Q)^m$
\begin{equation}\label{equ:6.18}
\frac12\int_{0}^{T}\| \overline{v} \|_{L^2(0,1)^m}^2 dt+\frac{1}{2\varepsilon}\|\overline{Y}(T)\|_{L^2(0,1)^n}^2  \leqslant \frac12\int_{0}^{T}\| v\|_{L^2(0,1)^m}^2 dt+\frac{1}{2\varepsilon}\|Y^{\overline{\widetilde{Y}},v}(T)\|_{L^2(0,1)^n}^2.
\end{equation}
This means that $\overline{v}$ minimizes $J_{\varepsilon,\overline{\widetilde{Y}}}$. Consequently $K_{\varepsilon}(\overline{\widetilde{Y}})=\overline{Y}$, Hence the continuity of $K_{\varepsilon}$,
Thus, the following result is then proved.

\begin{theorem}\label{theo6.1}
	Assume \eqref{condionalgsemi} is fulfilled. For all $Y_0$ in $(H^1_a)^n$ the  semi-linear parabolic degenerate system  \eqref{semilinsyst} is approximatively null controllable. i.e. for all $\varepsilon>0$ there exists a control $v_{\varepsilon} \in L^2(Q)^m$ for which the associated solution $Y^{v_{\varepsilon}}$ satisfies 
	\begin{equation}\label{theo1equ:6.19}
	\|Y^{v_{\varepsilon}}(T)\|_{L^2(0,1)^n}\leqslant \varepsilon.
	\end{equation}
	Moreover, there exists a positive constant $C>0$ such that 
	\begin{equation*}
	\|v_{\varepsilon}\|^2_{L^2(Q)^m} \leqslant C \|Y_0\|^2_{L^2(0,1)^n}
	\end{equation*}
\end{theorem}

From this theorem, we deduce following result 

\begin{theorem}\label{theo6.2}
		Assume \eqref{condionalgsemi} is fulfilled. For all $Y_0$ in $(H^1_a)^n$ the  semi-linear parabolic degenerate system \eqref{semilinsyst} is null controllable. i.e. There exists a control $v$ in $ L^2(Q)^m$ for which the associated solution $Y^{v}$ satisfies  $$Y^{v}(T,x)=0,\qquad \forall x \in (0,1)$$
	Moreover, there exists a positive constant $C>0$ such that 
	\begin{equation*}
	\|v\|^2_{L^2(Q)^m} \leqslant C \|Y_0\|^2_{L^2(0,1)^n}
	\end{equation*}
\end{theorem}

\begin{proof}
From Theorem \ref{theo6.1} the set $\{v_{\varepsilon}, \varepsilon>0 \}$ is bounded in $ L^2(Q)^m$, thus it contains a sequence $(v_{\varepsilon}^n)$ that converges (weakly) in $ L^2(Q)^m$ to a limit $v_0$ that satisfies 
	\begin{equation*}
\|v_0\|^2_{L^2(Q)^m} \leqslant C \|Y_0\|^2_{L^2(0,1)^n}
\end{equation*}
The sequence 
$(Y^{v_{\varepsilon}^n})$ converges strongly to $(Y^{v_0})$ in $X_T$. Moreover $(Y^{v_0})$ is the solution to  \eqref{semilinsyst} with $v=v_0$.  So according to \eqref{theo1equ:6.19}, for all $x\in (0,1)$ we have 
$$ Y^{v_0}(T,x)=0$$
thus,  the  semi-linear parabolic degenerate system \eqref{semilinsyst} with regular initial data is null controllable.
\end{proof}

Now we are able to give the proof of the null controllability of  the  semi-linear parabolic degenerate system \eqref{semilinsyst} with  general initial data.
 As in \cite{CanFrag,BOU,bahm}, we can show also the following well posedness of degenerate parabolic semi-linear systems which is of great utility. 

\begin{proposition}\label{propo6.3}
	 For all $Y_0$ in $L^2(0,1)^n$ the semi-linear system 
	 \begin{equation}\label{semilinsyst0}
	 \left\lbrace \begin{array}{lll}
	 \partial_t U=\mathbf{D}\mathcal{M}U + F(U)& in & Q, \\
	 \mathbf{C} U=0   & on & \Sigma, \\
	 U(0,x)=Y_0(x)\,\,  & in & (0,1),
	 \end{array}
	 \right.
	 \end{equation}
	 admits a solution $U$ in $X_T$.
\end{proposition}

\begin{theorem}
	For all $Y_0$ in $L^2(0,1)^n$ the semi-linear system parabolic degenerate system \eqref{semilinsyst}is null controllable.
\end{theorem}

\begin{proof}
	By Proposition 	\ref{propo6.3} the system \eqref{semilinsyst0} in the set $(0,T/2)\times(0,1)$ with initial data $Y_0$ in $L^2(0,1)^n$ admits a solution $\widetilde{Y}$ in $X_{T/2}$. Thus, for $t_0\in (0,T/2)$  we have  $\widetilde{Y}(t_0)\in (H^1_a)^n$. Now let us consider 
	\begin{equation}\label{semilinsyst1}
	\left\lbrace \begin{array}{lll}
	\partial_t \widetilde{U}=\mathbf{D}\mathcal{M}\widetilde{U} + F(\widetilde{U})+B v \mathbbm{1}_{\omega}& in & Q_{t_0}, \\
	\mathbf{C} \widetilde{U}=0   & on & \Sigma, \\
	\widetilde{U}(t_0,x)=\widetilde{Y}(t_0)(x)\,\,  & in & (0,1),
	\end{array}
	\right.
	\end{equation} 
where $Q_{t_0}= (t_0,T)\times(0,1)$. Due to Theorem \ref{theo6.2}, there exists a control $v_1\in L^2(Q_{t_0})^m$ for which the system \eqref{semilinsyst1} admits a solution  $\widetilde{U}$ that satisfies $\widetilde{U}(T,x)=0$ for all $x\in(0,1)$. Now let us define $(Y,v)$ as follow :
$$Y=\begin{cases}
U \text{ in } [0,t_0]\\
\widetilde{U} \text{ in } [t_0,T]
\end{cases} $$
$$v=\begin{cases}
0 \text{ in } [0,t_0]\\
v_1 \text{ in } [t_0,T]
\end{cases} $$
$Y$ is then a solution of the system \eqref{semilinsyst} that satisfies $Y(T,x)=0$ for all $x\in(0,1)$	which completes the proof.	 
\end{proof}

\bibliographystyle{plain}

\begin{thebibliography}{99}

\bibitem{bahm}
\newblock E. M. Ait Benhassi, F. Ammar Khodja, A. Hajjaj and L. Maniar,
\newblock  Null controllability of degenerate parabolic cascade systems,
\newblock  Portugal. Math., \textbf{68} (2011), 345--367.	

\bibitem{hjjaj}
\newblock E. M. Ait Benhassi, F. Ammar Khodja, A. Hajjaj and L. Maniar,
\newblock  Carleman estimates and null controllability of coupled degenerate systems,
\newblock  Evol. Equ. Control Theory 2, \textbf{3} (2013), 441--459.

\bibitem{benfama}
\newblock  E. M. Ait Benhassi, M. Fadili and L. Maniar,
\newblock On Algebraic condition for null controllability of some coupled degenerate systems.
\newblock  Mathematical Control and Related Fields. {\bf (8)} (2018),doi:10.3934/mcrf.2019004.


\bibitem{A2} F. Ammar-Khodja,  A.  Benabdallah, M. Gonz\'alez-Burgos,  L.  de Teresa, 
\newblock   Recent results on the controllability of linear coupled parabolic problems : a survey,
\newblock  Mathematical Control and Related Fields, {\bf 1(3)} (2011),  267--306.

\bibitem{A3} F. Ammar-Khodja,  A.  Benabdallah, C. Dupaix, M. Gonz\'alez-Burgos, 
\newblock A generalization of the Kalman rank condition for time-dependent coupled linear parabolic systems,
\newblock Diff. Equ. Appl. {\bf1} (2009) 427--457

\bibitem{A4} F. Ammar-Khodja,  A.  Benabdallah,  C.  Dupaix,  M. Gonz\'alez-Burgos, 
\newblock  A Kalman rank condition for the localized distributed controllability of a class of linear parabolic systems,
\newblock J. Evol. Equ. {\bf9} (2009) 267--291.


\bibitem{A5} F. Ammar Khodja, A. Benabdellah and C. Dupaix,
\newblock Null-controllability for some reaction-diffusion systems with one control force,
\newblock  J. Math. Anal. Appl., 320 (2006), 928--943.



\bibitem{BOU}  F. Alabau-Boussouira, P. Cannarsa, G. Fragnelli, 
\newblock Carleman estimates for degenerate parabolic operators with application to nullcontrolability,
\newblock  J. evol. equ. {\bf6} (2006), 161--204.


\bibitem{cmp} M. Campiti, G. Metafune, and D. Pallara,
\newblock  Degenerate self-adjoint evolution equations on the unit interval,
\newblock  Semigroup Forum, \textbf{57} (1998), 1--36.

\bibitem{CanFrag}
\newblock   P. Cannarsa,  G. Fragnelli,
\newblock Null controllability of semilinear degenerate parabolic equations in bounded domains,
\newblock Electronic Journal of Differential Equations, \textbf{136} (2006), pp 1--20.


\bibitem{CaMaVa} 
\newblock  P. Cannarsa, P. Martinez and J. Vancostenoble,
\newblock Null controllability of degenerate heat equations,
\newblock Adv. Differential Equations, \textbf{10} (2005), 153--190.

\bibitem{CaMaVa1} 
\newblock P. Cannarsa, P. Martinez and J. Vancostenoble,
\newblock Carleman estimates for a class of degenerate parabolic operators, 
\newblock  SIAM J. Control Optim. 47, (2008), no. 1, 1--19  


\bibitem{CaMaVa2}
\newblock P. Cannarsa, P. Martinez, J. Vancostenoble,
\newblock  Global Carleman estimates for degenerate parabolic operators with applications,    \newblock Memoirs of the American Mathematical Society (2016), Vol. 239


\bibitem{de-Ca} 
\newblock P. Cannarsa and  L. de Teresa, 	
\newblock  Controllability of 1-d coupled degenerate parabolic equations,
\newblock Electronic Journal of Differential Equations, {\bf73} (2009),  1--21.

\bibitem{FadiliManiar}
\newblock   M. Fadili and L. Maniar,
\newblock  Null controllability of $n$-coupled degenerate parabolic systems with $m$-controls,
\newblock  J. Evol. Equ. (2017), 1--30.

\bibitem{FerGonTer} 
\newblock	E. Fernandez-Cara, M. Gonzalez-Burgos  and L. de Teresa,
\newblock \textit{ Controllability of linear and semilinear non-diagonalizable parabolic systems },
\newblock COCV {\bf 21} (2015) 1178–1204

\bibitem{Fursikov} 
\newblock  A. V. Fursikov and O. Y. Imanuvilov,
\newblock   Controllability of evolution equations,
\newblock  Lectures notes series 34, Seoul National University Research Center, Seoul, 1996.



\bibitem{GonTer} 
\newblock   M. Gonzalez-Burgos, L. De Teresa, 
\newblock  Controllability results for cascade systems of $m$-coupled parabolic PDEs by one control force,
\newblock  Port. Math. {\bf67} (2010),  91--113.

\bibitem{Gueye}
\newblock  M. Gueye,
\newblock   Exact boundary controllability of 1-D parabolic and hyperbolic degenerate equations,
\newblock   SIAM J. Control Optim.  {\bf52} (2014), 2037--2054.




\bibitem{Leb_Rob}
\newblock     G. Lebeau and L. Robbiano,
\newblock   Contr\^ole exact de l'\'equation de la chaleur,
\newblock Comm. in PDE \textbf{20}  (1995), 335--356.

\bibitem{Meyer}
\newblock  R. D. Meyer,
\newblock   Degenerate elliptic differential systems,
\newblock J. Math. Anal. Appl. \textbf{29} (1970),  436--442.

\bibitem{Oliveira} 
\newblock  Luiz Augusto F. de Oliveira, On reaction-diffusion systems,
\newblock   Electron. J. Differential Equations 24 (1998).

\bibitem{Zab}   J. Zabczyk,
\newblock   Mathematical Control Theory, 
\newblock  Birkh\"auser, Boston, 1995.


\bibitem{zua} E. Zuazua,
\newblock  Exact controllability for semilinear wave equations in one space
dimension,
\newblock  Annales de l’IHP, section C, tome 10, n 1, (1993), 109-129.


	
\end{thebibliography}

\end{document}